\documentclass[oneside,11pt]{amsart}
\usepackage{amscd}
\usepackage{amssymb}
\usepackage{graphicx}
\usepackage[usenames,dvipsnames]{color}

\allowdisplaybreaks[1]

\usepackage[all]{xy}

\makeatletter
\def\section{\@startsection{section}{1}%
  \z@{.7\linespacing\@plus\linespacing}{.5\linespacing}%
  {\normalfont\centering}}
\makeatother

\newtheorem{thm}{Theorem}[section]
\newtheorem{cor}[thm]{Corollary}
\newtheorem{lem}[thm]{Lemma}

\theoremstyle{definition}
\newtheorem{defn}[thm]{Definition}
\newtheorem{rem}[thm]{Remark}

\numberwithin{equation}{section}

\newcommand{\thmref}[1]{Theorem~\ref{#1}}

\newcommand{\lemref}[1]{Lemma~\ref{#1}}

\newcommand{\C}{\hbox{${\mathbb C}$}}

\newcommand{\Z}{\hbox{${\mathbb Z}$}}

\newcommand{\A}{\hbox{${\mathcal A}$}}

\newcommand{\al}{\alpha}

\newcommand{\h}{\hbox{${\mathfrak h}$}}

\newcommand{\be}{\begin{equation}}

\newcommand{\bmm}{\left[\begin{array}}
\newcommand{\emm}{\end{array}\right]}

\newcommand{\sltu}{\mathfrak{sl}_2(\C)}
\newcommand{\slhat}{\widehat{\mathfrak{sl}}_2(\C)}

\newcommand{\bZ}{\textbf{Z}}
\newcommand{\half}{\frac{1}{2}}
\newcommand{\nhalf}{\frac{-1}{2}}
\newcommand{\fourth}{\frac{1}{4}}

\newcommand{\xhalf}[1]{\frac{#1}{2}}
\newcommand{\nxhalf}[1]{\frac{-#1}{2}}
\newcommand{\tr}{\text{Tr}}
\newcommand{\dirac}{\delta\left(\frac{w}{z}\right)}

\newcommand{\Zalg}{\mathcal{Z}}

\newcommand{\com}[2]{\left[#1,#2\right]}
\newcommand{\gencom}[2]{\left[\left[#1,#2\right]\right]}
\newcommand{\anticom}[2]{\left\{#1,#2\right\}}
\newcommand{\lrpar}[1]{\left(#1\right)}
\newcommand{\lrbrace}[1]{\left\{#1\right\}}
\newcommand{\lrbrack}[1]{\left[#1\right]}
\newcommand{\ub}[1]{\underbar{$#1$}}
\newcommand{\ds}{\displaystyle}

\begin{document}


\title{Realization of $\slhat$ at the Critical Level}

\author[J. Dunbar]{Jonathan Dunbar}
	\address{Dept of Mathematics	\\
	Spring Hill College	\\	
	Mobile, AL 36608}
\author[N. Jing]{Naihuan Jing}
	\address{Dept of Mathematics \\
	North Carolina State University\\
	Raleigh, NC 27695-8205}
\author[K. C. Misra]{Kailash C. Misra}
	\address{Dept of Mathematics \\
	North Carolina State University\\
	Raleigh, NC 27695-8205}


\begin{abstract}
An explicit realization of the affine Lie algebra $\slhat$ at the critical level is
constructed using a mixture of bosons and parafermions. Subsequently
a representation of the associated Lepowsky-Wilson $\bZ$-algebra is given on
a space of the tensor product of bosonic fields and certain semi-infinite wedge products.
\end{abstract}



\maketitle

\section{\bf{Introduction}}\label{Introduction}

Affine Lie algebras form an important class of Kac-Moody Lie algebras (cf. \cite{Kac90}). The first nontrivial construction of a representation of $\slhat$ was discovered in terms of certain differential operators  \cite{LepWil78} called vertex operators acting on a Fock Space. This explicit realization of affine Lie algebras by vertex operators initiated a flurry of activities which led to many important connections of affine Lie algebra representations with other areas of mathematics and physics. It also led to the discovery of new algebraic structures such as vertex operator algebras (cf. \cite{B, FLM, LepLi}).

In  \cite{LepWil81, LepWil84} Lepowsky and Wilson introduced a certain nonassociative algebra called a $\bZ$-algebra associated with an integrable highest-weight module of an affine Lie algebra and a suitable infinite-dimensional Heisenberg subalgebra. In particular, the $\bZ$-algebra is generated by certain operators called $\bZ$-operators centralizing the action of the Heisenberg subalgebra on the representation space,  hence acting on a special subspace called the vacuum space. It was shown in \cite{LepWil84} that the representation space is irreducible if and only if the vacuum space is irreducible as a $\bZ$-algebra module. In \cite{LepPrimc}, Lepowsky and Primc  used the $\bZ$-algebra theory to give explicit constructions of all integrable highest weight $\slhat$-modules at positive integral level in the homogeneous gradation. Such $\bZ$-algebras are essentially equivalent to parafermion algebras as explained in \cite{DL}. The construction was extended to that of the quantum affine algebra $U_q(\slhat)$ in \cite{Jing96}.

On the other hand, Wakimoto \cite{Wak} gave a free field realization of the affine Lie algebra $\slhat$ at an arbitrary level. Subsequently, this construction was extended to other affine Lie algebras  by Feigin and Frenkel \cite{FeigFrenk, FeigFrenk90b}. Geometric realizations of the Wakimoto representation at the critical level were given in \cite{FG}. For a generic weight $\lambda$ at the critical level $-2$ the graded dimension of the simple $\slhat$-module $L(\lambda)$ is known (cf. \cite{KK}, \cite{FF}, \cite{Ku}, \cite{Mal}).

Motivated by explicit constructions in \cite{LepPrimc, Jing96} at the positive levels, we present an explicit realization of the affine Lie algebra
$\slhat$ at the critical level $-2$. The corresponding representation involves an interesting Clifford-type (or parafermion) algebra and is  simpler  than the Wakimoto module. It would be interesting to compare our new construction with the level $2$ $\slhat$-modules constructed in  \cite{Jing96}. Finally, we give a representation of the associated Lepowsky-Wilson $\bZ$-algebra and determine its graded dimension. It is worthwhile to note that this $\bZ$-algebra realization at the critical level $-2$ is different from the general construction of \cite{LepPrimc} for positive levels.

\section{\bf{Preliminaries}}\label{Prelim}
The simple Lie algebra $\sltu$ is generated by $X,\ Y,$ and
$H$,  where
$$X=\left[\begin{array}{cc}0&1\\ 0&0\end{array}\right],\
Y=\left[\begin{array}{cc}0&0\\ 1&0\end{array}\right],\
H=\left[\begin{array}{cc}1&0\\ 0&-1\end{array}\right].$$

The roots of $\sltu$ with Cartan subalgebra $\C{H}$ are $\lrbrace{\al,-\al}$ with $\al(H) = 2$.

The affine Lie algebra $\slhat = \sltu\otimes\C[t,t^{-1}]\oplus\C c\oplus\C d,$ where $c$ is the central element and $d = 1\otimes t\frac{d}{dt}$ is the degree derivation.
For any $a\in\sltu$ and $m\in\Z$, we denote $a(m) = a\otimes t^m\in\slhat$. For any $a,b\in\sltu$ and $m,n\in\Z$, the commutation relations in $\slhat$ are given by:

\begin{equation}\label{componentCom}
\begin{cases}
\com{a(m)}{b(n)} = \com ab(m+n)+\tr(ab)mc\delta_{m,-n},\\	
\com{c}{\slhat} = 0,	\hspace{.5in}	\com{d}{a(m)}=ma(m).
\end{cases}
\end{equation}

The affine Lie algebra $\slhat$ is generated by $h_0, h_1, e_0, e_1, f_0, f_1$ and $d$ where:
$$\begin{array}{ll}
e_0 = Y(1),	&	e_1 = X(0),	\\
f_0 = X(-1),	&	f_1 = Y(0)	,	\\
h_0 = -H(0) + c,	&	h_1 = H(0).
\end{array}$$

Let $\ds\h=\big(\bigoplus\limits_{n\in\Z} {\mathbb C}H(n)\big) \oplus \C c \oplus \C d$ and
$\ds\h_0={\mathbb C}H(0)\oplus \C c \oplus \C d$ be the affine Cartan subalgebra.
Then, for $\lambda\in\h_0^*$, $\lambda(c)$ is called the level of $\lambda$, and for $\slhat$, $\lambda(c) = -2$ is the critical level. We consider the Heisenberg subalgebra $\ds\h'=\big(\bigoplus\limits_{n\ne0} {\mathbb C}H(n)\big)\oplus \C c$ of $\slhat$ and set  $\ds\h'^\pm=\bigoplus\limits_{n>0} {\mathbb C}H(\pm n)$.

It follows by  \eqref{componentCom} that $\com{H(m)}{H(n)} = 2mc\delta_{m,-n}$. Throughout this paper, we use the notations and settings of \cite{FLM} and \cite{LepLi}.  We define the action of the Heisenberg subalgebra $\h'$ on the symmetric algebra $S(\h'^-)$ as follows:

\begin{align*}
c\cdot v		&=	-2v,	\\
H(-m)\cdot v	&=	H(-m)v,	\\
H(m) \cdot v	&=	-4m\partial_{H(-m)}(v),
\end{align*}
for all $v\in S(\h'^-)$, $m>0$, where $\partial_{H(-m)}$ denotes the formal partial derivative with respect to $H(-m)$.

Let $a(z) = \sum\limits_{m\in\Z}a(m)z^{-m}$, for all $a\in\sltu$. We call $a(m)$ the $m^{\text{th}}$ \textit{component} of $a(z)$. Let $\delta(z) = \sum\limits_{n\in\Z}z^n$ be the formal delta function.  Then using \eqref{componentCom} we have:

\begin{align}
	\com{H(z)}{X(w)}	&=	2X(w)\dirac, \label{hxbracket}	\\
	\com{H(z)}{Y(w)}	&=	-2Y(w)\dirac, \label{hybracket}	\\
	\com{X(z)}{Y(w)}	&=	H(w)\dirac + cw\partial_w\dirac.\label{xybracket}
\end{align}

Now let $V$ be any $\slhat$-module induced from that of the Heisenberg subalgebra $\mathfrak h'$, we define the exponential operators on $V$:
$$E^\pm_+(z) = \exp\lrpar{\mp\sum\limits_{n>0}\frac{H(-n)}{2n}z^n}\hspace{.2in} \text{and}
\hspace{.2in}E^\pm_-(z) = \exp\lrpar{\pm\sum\limits_{n>0}\frac{H(n)}{2n}z^{-n}}.$$

The following identities are straightforward.

\begin{lem}$($cf. \cite{LepWil84}$)$
\begin{align*}
E^+_{\pm}(z)E^-_{\pm}(z) &=	1,	\\
E^\pm_-(z)E^\pm_+(w)	&=	E^\pm_+(w)E^\pm_-(z)\lrpar{1-\frac{w}{z}}^{-1}	,\\
E^+_-(z)E^-_+(w)	&=	E^-_+(w)E^+_-(z)\lrpar{1-\frac wz},	\\
E^+_\pm(z)E^+_\pm(w)	&=	E^+_\pm(w)E^+_\pm(z),	\\
E^-_\pm(z)E^+_\pm(w)	&=	E^+_\pm(w)E^-_\pm(z),	\\
E^-_\pm(z)E^-_\pm(w)	&=	E^-_\pm(w)E^-_\pm(z),	\\
\partial_z\left(E^\pm_+(z)E^\pm_-(z)\right)	
		&=	E^\pm_+(z)\lrpar{\mp\sum_{n\ne0}\frac{H(n)}{2}z^{-n-1}}E^\pm_-(z).\end{align*}
\end{lem}
\begin{proof}
The proof of all properties except the last one are similar to the corresponding results in \cite{LepWil84}.  We include the proof of the last property below.

\begin{align*}
\partial_z\lrpar{E^\pm_+(z)E^\pm_-(z)}	
	&=	E^\pm_+(z)\partial_z\lrbrack{\mp\sum\limits_{n>0}\frac{H(-n)}{2n} z^n}E^\pm_-(z) + E^\pm_+(z)\partial_z\lrbrack{\pm\sum\limits_{n>0}\frac{H(n)}{2n} z^{-n}}E^\pm_-(z)	\\
	&=	E^\pm_+(z)\lrbrack{\mp\sum\limits_{n>0}\frac{H(-n)}{2} z^{n-1}}E^\pm_-(z) + E^\pm_+(z)\lrbrack{\mp\sum\limits_{n>0}\frac{H(n)}{2} z^{-n-1}}E^\pm_-(z)	\\
	&=	E^\pm_+(z)\lrbrack{\mp\sum\limits_{n\ne0}\frac{H(n)}{2} z^{-n-1}}E^\pm_-(z)
\end{align*}
\end{proof}

Now we define the $\bZ$-operators on $V$ as follows:
\begin{align*}
		\bZ^+(z)	=	\bZ(\al,z)	=	E^-_+(z)X(z)E^-_-(z),	\\
		\bZ^-	(z)	=	\bZ(-\al,z)	=	E^+_+(z)Y(z)E^+_-(z).
\end{align*}

Let $\bZ^\pm(z) = \sum\limits_{m\in\Z}\bZ^\pm(m)z^{-m}$. The $\bZ$-algebra $\Zalg$ is generated by the homogeneous components $\bZ^\pm(m)$, for all $m\in\Z$. As shown in \cite{LepWil84}, $V$ has the following tensor product decomposition:
$$V \simeq S(\h'^-)\otimes\Omega(V),$$
where $\Omega(V)=\lrbrace{v\in V\vert\h'^+\cdot v = 0}$ is called the vacuum space. The following lemma can be proved using standard techniques, as in \cite{LepWil84}.

\begin{lem}\label{H commutes with Z}$($cf. \cite{LepWil84}$)$
	The $\bZ$-operators $\bZ^\pm(m)$ commutes with the $H(n)$'s and hence with the Heisenberg algebra $\h'$.
\end{lem}

As a consequence of \lemref{H commutes with Z}, we have
	\begin{align*}
	\com{E^\pm_+(z)}{\bZ^+(w)}=0=\com{E^\pm_-(z)}{\bZ^+(w)}\ 	\\	
	\com{E^\pm_+(z)}{\bZ^-(w)}=0=\com{E^\pm_-(z)}{\bZ^-(w)}.\end{align*}

By \lemref{H commutes with Z}, the $\bZ$-operators centralize the action of the Heisenberg subalgebra and hence act on the vacuum space $\Omega(V)$. Thus $\Omega(V)$ is a $\Zalg$-module. As shown in \cite{LepWil84}, $V$ is irreducible if and only if $\Omega(V)$ is irreducible as a $\Zalg$-module.

For $\phi_1,\ \phi_2 = \pm\al$, we define the generalized commutator bracket as follows:
\begin{align}\label{GenComDef}
&\gencom{\bZ(\phi_1,z)}{\bZ(\phi_2,w)} \\
&	\qquad\qquad = \bZ(\phi_1,z)\bZ(\phi_2,w)\lrpar{1-\frac wz}^{\frac{\lrpar{\phi_1,\phi_2}}{-2}}
	- \bZ(\phi_2,w)\bZ(\phi_1,z)\lrpar{1-\frac zw}^{\frac{\lrpar{\phi_1,\phi_2}}{-2}}.\nonumber
\end{align}

The following Lemma follows easily using \lemref{H commutes with Z}, as in \cite{LepWil84}.

\begin{lem}\label{Zalg relations}$($cf. \cite{LepWil84}$)$ The following relations hold:
\begin{equation}
\begin{cases}
	\gencom{\bZ^\pm(z)}{\bZ^\pm(w)}	&= 0,		\\
	\gencom{\bZ^+(z)}{\bZ^-(w)}		&= H(0)\dirac - 2w\partial_w\dirac.	
\end{cases}
\end{equation}\label{Z-relations}
\end{lem}

\section{\bf{Representation of $\slhat$ at the Critical Level}}\label{Rep}
		
We define the Clifford-like algebra $\A$ generated by $\lrbrace{A(m), A^*(m)\vert m\in\Z+\half}$ subject to the relations:
\begin{equation}\label{ARelation1}
	\anticom{A(m)}{A^*(n)} = -\lrpar{m^2-\fourth}\delta_{m+n,0},
\end{equation}
\begin{equation}\label{ARelation2}
	\anticom{A(m)}{A(n)} = 0 = \anticom{A^*(m)}{A^*(n)},
\end{equation}
where $\anticom{\ }{\ }$ denotes the anti-commutator bracket.

For $a \in \{A, A^*\}$ we define the formal generating series
	$\ds a(z) = \sum_{n\in\Z+\half} a(n)z^{-n-\half}$.
	

\begin{defn}\label{Normal Ordering}
For $m,n\in\Z+\half$, $a, b \in \{A, A^*\}$ we define the normally ordered products
$$:\!a(m)b(n)\!:\ = \left\{\begin{array}{rr}a(m)b(n), & m< 0 \\ -b(n)a(m), & m>0
\end{array}\right..$$

The normally ordered products for the generating series  are given by

$$:\!\!a(z)b(w)\!\!: = \hspace{-.1in}\sum\limits_{m,n\in\Z+\half} \hspace{-.1in} :\!\!a(m)b(n)\!\!:z^{-m-\half} w^{-n-\half}$$
for $a, b \in \{A, A^*\}$ and we have $:\!\!a(z)b(w)\!\!: = -:\!\!b(w)a(z)\!\!:$. Therefore, $:\!a(w)a(w)\!:\ = 0$ for $a \in \{A, A^*\}$.

\end{defn}

For $a, b \in \{A, A^*\}$ we define the contraction function $\underbrace{a(z)b(w)}\! = a(z)b(w) - :\!a(z)b(w)\!:\!.$
Then by \eqref{ARelation1}, and \eqref{ARelation2}  we have
\begin{equation}\label{aastarcontract}
\begin{cases}
	\underbrace{A(z)A(w)} 	= 	0	= \underbrace{A^*(z)A^*(w)},\\	
	\underbrace{A(z)A^*(w)}	=	\ds\frac{-2zw}{(z-w)^3}  =  \underbrace{A^*(z)A(w)}
\end{cases}
\end{equation}

We consider a natural representation of the algebra $\A$ as follows. The exterior algebra $\Lambda(V)$ of the infinite-dimensional
vector space $V=\oplus_{i\in\mathbb Z+1/2}\mathbb C\ub{i}$ is spanned by the finite wedge products $u=\ub{i_1}\wedge\ldots\wedge \ub{i_k}$, where $i_1<\cdots<i_k$, $i_j\in\mathbb Z+1/2$. The set of
indices appearing in $u$ is called the support of $u$, i.e., $supp(u)=\{i_1, \ldots, i_k\}$. We now consider the
semi-infinite wedge space $\Lambda:=\Lambda^{\infty/2}(V)$, which is an infinite-dimensional $\Lambda(V)$-module spanned by
the semi-infinite wedge products (of weight $r$)
\begin{equation}\label{wedgeprod}
v = \ub{i_1}\wedge\ub{i_2}\wedge\ldots\wedge\ub{i_k}\wedge\ldots,
\end{equation}
where the support $supp(v)=\{i_1, i_2, \cdots \}$ satisfy the condition that
$i_1<i_2<\cdots<i_k<\cdots$ and $i_k=k+r+1/2$ for a fixed $r\in\mathbb Z$ and sufficiently large integer $k$.
Here $\Lambda(V)$ acts on $\Lambda^{\infty/2}(V)$ by concatenation.
A general homogeneous element $v$ in $\Lambda$ is characterized by the
condition that both $supp(v)\cap(\mathbb Z_-+1/2)$ and $(\mathbb Z_++1/2)-supp(v)$ are finite subsets.
Here $\mathbb Z_{+}$
(or $\mathbb Z_{-}$)
denotes the set of non-negative (or negative) integers respectively.

The semi-infinite wedge products satisfy similar properties as finite wedge products such as adjacent factors anti-commute with each other:
	\begin{equation}\label{antisymm1}
		\ub i\wedge\ub j\wedge\ldots = - \ub j\wedge\ub i\wedge\dots.\end{equation}
\noindent It follows that $\ub{i}\wedge\ub{i}\wedge\ldots = 0$, for any $\ub{i}$. Also, it follows that the factors of $v$ or any similar semi-infinite wedge product can always be placed into ascending numerical order, as described above in the definition of $v$ (see Eq. (\ref{wedgeprod})).

Since, for $n_j\in\Z_++\half$, only finitely many ${n_j}$ are not in $supp(v)$, it may then be convenient to use the notation $\ub{\widehat{i}}$ to imply that $\ub i$ is not a factor of $v$. For each $i\in\mathbb Z+1/2$ we define the
operator $\partial_{\ub{i}}$ acting on $\Lambda$ as follows. For any semi-infinite wedge product $v$ as in Eq. (\ref{wedgeprod}), if $i\notin supp(v)$ then $\partial_{\ub{i}}v=0$. Otherwise, if $i_k=i$
\begin{equation}\label{antisymm2}
\partial_{\ub{i}}(\ub{i_1}\wedge \cdots \wedge \ub{i_k}\wedge\cdots)=(-1)^{k+1}\ub{i_1}\wedge\ub{i_2}\wedge\dots\wedge\ub{\widehat i_k}\wedge\dots.
\end{equation}
The action of $\partial_{\ub{i}}$ is well-defined. In fact, without loss of generality
suppose we rewrite the semi-infinite wedge product
$v=\ub{i_1}\wedge \cdots \wedge \ub{i_k}\wedge\cdots=-\ub{i_1}\wedge \cdots \wedge \ub{i_k+1}\wedge\ub{i_k}\cdots$.
Since $i_k=i$ is now in the $(k+1)$th position, our rule (\ref{antisymm2}) gives that
$$\partial_{\ub{i}}(v)=\partial_{\ub{i}}(-\ub{i_1}\wedge \cdots \wedge \ub{i_{k+1}}\wedge \ub{i_{k}}\wedge\cdots)=-(-1)^{k+2}\ub{i_1}\wedge\ub{i_2}\wedge\dots\wedge\ub{\widehat i_k}\wedge\dots
$$
which is the same as the right-hand side of (\ref{antisymm2}).

An example of homogeneous semi-infinite wedge vector $v$ is
$$v = \ub{-9/2}\wedge\ub{-3/2}\wedge\ub{-1/2}\wedge\ub{1/2}\wedge\ub{5/2}\wedge\ub{7/2}\wedge\ub{9/2}\wedge\ldots,$$
which has $supp(v) = 
\{-9/2, -3/2, -1/2\}\cup(\mathbb Z_++1/2)-\{3/2\}$.

Equivalently, we could write this example as
$$v = \ub{-9/2}\wedge{\widehat{\ub{-7/2}}}\wedge{\widehat{\ub{-5/2}}}\wedge\ub{-3/2}\wedge\ub{-1/2}\wedge\ub{1/2}\wedge{\widehat{\ub{3/2}}}\wedge\ub{5/2}\wedge\ub{7/2}\wedge\ub{9/2}\wedge\ldots ,$$
\text{ or}
$$v = -\partial_{\ub{-5/2}}\lrpar{\ub{-9/2}\wedge\ub{-5/2}\wedge\ub{-3/2}\wedge\ub{-1/2}\wedge\ub{1/2}\wedge\ub{5/2}\wedge\ub{7/2}\wedge\ub{9/2}\wedge\ldots}.$$

Now, define $\Lambda_{\nhalf,\widehat\half}$ to be the subspace of $\Lambda$ spanned by the vectors of the form
$$v = \ub{m_r}\wedge\ldots\wedge\ub{m_1}\wedge\ub{-1/2}\wedge\widehat{\ub{1/2}}\wedge\ub{n_1}\wedge\ub{n_2}\wedge\ldots.$$
In particular, we stress that for these vectors $v$, ${-1/2}\in supp(v)$ and ${1/2}\notin supp(v)$. Using the notation described above to highlight the factors omitted from $supp(v)$, we can write any vector in $\Lambda_{\nhalf,\widehat\half}$ as a linear combination of vectors of the form:
\begin{equation}\label{vector form}
	v = \ub{m_r}\wedge\dots\wedge \ub{m_1}\wedge \ub{-1/2} \wedge\ub{\widehat{1/2}}\wedge
	\dots\wedge \ub{\widehat{n_1}}\wedge\dots\wedge\ub{\widehat{n_s}}\wedge\dots,
\end{equation}
	where $m_r<\ldots<m_1<\nhalf<n_1<\ldots<n_s$, for some positive integers $r$ and $s$. Note especially that the important characteristic of this $v$ is that the negative half-integers $m_i's$ are included in $supp(v)$ and the positive half-integers $n_i's$ are not included in $supp(v)$.
Now for a homogeneous element $v\in\Lambda_{\nhalf,\widehat\half}$, as in  \eqref{vector form}, we define the degree of $v$ by
\begin{equation}\label{cliff alg deg}
deg(v) = \sum\limits^r_{i=1}\lrpar{-m_i-\half} + \sum\limits^s_{j=1}\lrpar{n_j - \half}.
\end{equation}

For example, let
$$w = \ub{-11/2}\wedge\ub{-5/2}\wedge\ub{-3/2}\wedge\ub{-1/2}\wedge\ub{3/2}\wedge\ub{7/2}\wedge\ub{11/2}\wedge\ub{15/2}\wedge\ub{17/2}\wedge\ldots.$$
Then,
$$\begin{array}{rll}
	\ds deg(w)	&=\ds	\lrpar{\xhalf{11}-\half} + \lrpar{\xhalf{5}-\half} + \lrpar{\xhalf{3}-\half}	&	\\
			&\ds\qquad	+ \lrpar{\xhalf{5}-\half} + \lrpar{\xhalf{9}-\half} + \lrpar{\xhalf{13}-\half}	&=	 20.\end{array}$$

Define the action of the algebra $\A$ on $\Lambda_{\nhalf,\widehat\half}$ as follows:
		\begin{align}
		A(m) \cdot v	&=	\lrpar{m-\half}\ub{m}\wedge v,\\
		A^*(m) \cdot v	&=	\lrpar{m-\half}\partial_{\ub{-m}}\lrpar v,
		\end{align}
where $v = \ub{m_r}\wedge\dots\wedge \ub{m_1}\wedge \ub{-1/2} \wedge\ub{\widehat{1/2}}\wedge
	\dots\wedge \ub{\widehat{n_1}}\wedge\dots\wedge\ub{\widehat{n_s}}\wedge\dots$ in $\Lambda_{\nhalf,\widehat\half}$, and $A(m),\ A^*(m)\in \A$. This action is extended by linearity to all of $\Lambda_{\nhalf,\widehat\half}$.

\begin{thm}
	Under the above action $\Lambda_{\nhalf,\widehat\half}$ is an $\A$-module.
\end{thm}

\begin{proof}
It is sufficient to show that the action preserves relations \eqref{ARelation1}  and \eqref{ARelation2} for\newline $v = \ub{m_r}\wedge\dots\wedge \ub{m_1}\wedge \ub{-1/2} \wedge\ub{\widehat{1/2}}\wedge\dots\wedge\ub{\widehat{n_1}}\wedge\dots\wedge\ub{\widehat{n_s}}\wedge\dots\in\Lambda_{\nhalf,\widehat\half}$. First, we prove that \eqref{ARelation1} holds:
	$$\anticom{A(m)}{A^*(n)}\cdot v = -\lrpar{m^2-\fourth}\delta_{m+n,0}v.$$

For $m,n\in\Z+\half$, there are four cases to consider. Namely,
	\begin{itemize}
	\item[(i)] ${m},{-n}\in\ supp(v)$,
	\item[(ii)] ${m} \notin supp(v),{-n}\in\ supp(v)$,
	\item[(iii)] ${m}\in supp(v),{-n}\notin\ supp(v)$, and
	\item[(iv)] ${m},{-n}\notin\ supp(v)$.
	\end{itemize}
	We only prove that \eqref{ARelation1} holds in case (i) below since the other three cases are similar. In this case, we have
	\begin{align*}
		\anticom{A(m)}{A^*(n)}\cdot v	
			&=	A(m)\cdot\lrpar{\lrpar{n-\half}\partial_{\ub {-n}}\lrpar v}
				+	A^*(n)\cdot\lrpar{\lrpar{m-\half}\ub{m}\wedge v}	\\
			&=	A(m)\cdot\lrpar{\lrpar{n-\half}\partial_{\ub {-n}}\lrpar v}	\\
			&=	\lrpar{m-\half}\lrpar{n-\half}\ub{m}\wedge\lrpar{\partial_{\ub {-n}}
				\lrpar v}.
	\end{align*}
			
Clearly, $\ub{m}\wedge\lrpar{\partial_{\ub {-n}}\lrpar v}\ne0$ if and only if $m= -n$. Hence, when ${m},{-n}\in supp(v)$, we have $\anticom{A(m)}{A^*(n)}\cdot v = -\lrpar{m^2-\fourth}v\delta_{m,-n}$ proving \eqref{ARelation1}.

Now to show that $\anticom{A(m)}{A(n)}\cdot v = 0$ holds, we
observe that
	\begin{align*}
		\anticom{A(m)}{A(n)}\cdot v	
		  &=	A(m)A(n)\cdot v + A(n)A(m)\cdot v\\
			&=	A(m)\cdot\lrpar{\lrpar{n-\half}\ub{n}\wedge v}
				+	A(n)\cdot\lrpar{\lrpar{m-\half}\ub{m}\wedge v}	\\
			&=	\lrpar{m-\half}\lrpar{n-\half}\ub{m}\wedge\ub{n}\wedge v
				+	\lrpar{n-\half}\lrpar{m-\half}\ub{n}\wedge\ub{m}\wedge v	\\
			&= 0,
	\end{align*}
	since $\ub{m}\wedge\ub{n} = -\ub{n}\wedge\ub{m}$.

It can be shown similarly using relation \eqref{antisymm2} that $\anticom{A^*(m)}{A^*(n)}\cdot v = 0$ also holds.
\end{proof}

Consider the group algebra $\C\lrbrack{\Z\al}$ where, for all $\beta, \lambda\in\Z\al$, we define the actions
                 \begin{equation}
		\begin{cases}
			e^\beta \cdot e^\lambda	=	e^{\beta+\lambda}, \\ 
			d\cdot e^\lambda 		=	-\fourth(\lambda,\lambda)e^\lambda.
			\label{gp alg deg}
			\end{cases}
		\end{equation}
For any coweight $\beta(0)$ ($\beta\in \Z\al$) we define
\begin{equation*}
z^{\beta(0)} \cdot e^\lambda	=	z^{(\beta,\lambda)}e^{\lambda},
\end{equation*}
and we identify $\alpha(0)$ with $H(0)$.

\begin{thm}
\label{Rep Defn}
	Let $V = S(\h^-)\otimes\Lambda_{\nhalf,\widehat\half}\otimes\C\lrbrack{\Z\al}$.
	Define the map $\pi:\slhat\longrightarrow gl(V)$ by

\begin{align*}
	X(z)	&\longmapsto	E^+_+(z)E^+_-(z)	\otimes	A(z)\otimes e^\al z^{-\frac{\al(0)}{2}}, \\
	Y(z)	&\longmapsto	E^-_+(z)E^-_-(z)	\otimes	A^*(z)\otimes e^{-\al} z^{\frac{\al(0)}{2}} ,\\
	H(z)	&\longmapsto	H^*(z)\otimes 1	\otimes	1+1\otimes 1\otimes H(0),	\\
	c	&\longmapsto	-2,	\\
	d	&\longmapsto	d \otimes1\otimes1 + 1\otimes d\otimes1 + 1\otimes1\otimes d,
\end{align*}
where $H^*(z)=H(z)-H(0)$.  Then $\pi$ is a representation of $\slhat$ on $V$.
\end{thm}

\begin{proof}
To see that $\pi$ defines a representation of $\slhat$, we simply need to show that the map $\pi$ preserves relations \eqref{hxbracket} -- \eqref{xybracket}. For example, using (\ref{aastarcontract}) we have

\begin{align*}
\pi(X(z))\pi(Y(w))	&=	E^+_+(z)E^+_-(z)E^-_+(w)E^-_-(w)	\otimes
			A(z)e^\al z^{-\frac{\al(0)}{2}}A^*(w)e^{-\al} w^{\frac{\al(0)}{2}}	\\
		&=	E^+_+(z)E^-_+(w)E^+_-(z)E^-_-(w)	\otimes
			:\!A(z)A^*(w)\!:\lrpar{z-w}z^{-\frac{\al(0)}{2}}w^{\frac{\al(0)}{2}}	\\
		&\qquad	+	E^+_+(z)E^-_+(w)E^+_-(z)E^-_-(w)	\otimes	
			\lrpar{\frac{-2zw}{(z-w)^2}}\lrpar{\frac{w}{z}}^{\frac{\al(0)}{2}}, 
\end{align*}

and
\begin{align*}
\com{\pi(X(z))}{\pi(Y(w))}	
			&=	-2E^+_+(z)E^-_+(w)E^+_-(z)E^-_-(w)
				\lrpar{\frac{w}{z}}^{\frac{\al(0)}{2}}\lrbrack{w\partial_w\dirac}	\\
			&=	- 2w\partial_w\lrbrack{E^+_+(z)E^-_+(w)E^+_-(z)E^-_-(w)
				\lrpar{\frac{w}{z}}^{\frac{\al(0)}{2}}\dirac}	\\
			&\qquad	+ 2w\partial_w\lrbrack{E^+_+(z)E^-_+(w)E^+_-(z)E^-_-(w)}
				\lrpar{\frac{w}{z}}^{\frac{\al(0)}{2}}\dirac	\\
			&\qquad	+ 2wE^+_+(z)E^-_+(w)E^+_-(z)E^-_-(w)
				\partial_w\lrbrack{\lrpar{\frac{w}{z}}^{\frac{\al(0)}{2}}}\dirac	\\
			&=	- 2w\partial_w\dirac + \sum_{n\ne0}H(n)w^{-n} \dirac + \al(0)\dirac	\\
			&=	\sum_{n\ne0}H(n)w^{-n} \dirac + H(0)\dirac- 2w\partial_w\dirac		\\
			&=	H(w)\dirac - 2w\partial_w\dirac = \pi(\com{X(z)}{Y(w)}).
\end{align*}

It follows similarly that the map $\pi$ preserves the relations \eqref{hxbracket} and \eqref{hybracket}.

Therefore, $\pi$ is a representation of $\slhat$ on $V$.
\end{proof}

Now, with the degree defined in (\ref{cliff alg deg}, \ref{gp alg deg}) the following result is immediate.
\begin{cor}\label{character of V}
The graded $q$-dimension of $V$ is
	$\ds \dim_q(V) = \begin{array}{c}
			\prod\limits_{m\ge0}\lrpar{1+q^m}^2\sum\limits_{p\in\Z}q^{\frac{p^2}{2}}	\\
			\hline	\prod\limits_{n\ge1}\lrpar{1-q^{n}}\end{array}.$
\end{cor}

\begin{rem}
It is known (cf. \cite{KK, Ku, Mal}) that the graded $q$-dimension of an irreducible $\slhat$ highest weight module $L(\lambda)$ of level $-2$ is
$\ds \dim_q(L(\lambda)) = \prod_{n\ge0}\lrpar{1+q^n}^2.$
\end{rem}

\begin{thm}\label{v is a HWV}\ \newline
	Consider  $v_0 = 1\otimes\ub{-1/2}\wedge\ub{3/2}\wedge\ub{5/2}\wedge\dots\otimes1$, and
		$v_1 = 1\otimes\ub{-1/2}\wedge\ub{3/2}\wedge\ub{5/2}\wedge\dots\otimes e^{-\al}$ in $V$.
	Then, $v_0$ and $v_1$ are highest weight vectors with weights
	$\lambda_0 = -2\Lambda_0$ and \newline
	$\lambda_1 = -2\Lambda_1-\half\delta$, respectively.
\end{thm}

\begin{proof}
	Recall that $e_0 = Y(1)$ and $e_1 = X(0)$. 	Then,
	
	\begin{align*}
	Y(z)\cdot v_0	&=	E^-_+(z)\cdot1\otimes\sum\limits_{n\in\Z+\half}A^*(n)\cdot
					\lrpar{\ub{-1/2}\wedge\ub{3/2}\wedge\ub{5/2}\wedge\dots}
					z^{-n-\half}\otimes e^{-\al}z^{\frac\al2}\cdot1	\\
				&=	E^-_+(z)\otimes\sum\limits_{n\le\nxhalf3}\lrpar{n-\half}
					\partial_{\ub {-n}}\lrpar{\ub{-1/2}\wedge\ub{3/2}\wedge\ub{5/2}
					\wedge\dots}\otimes e^{-\al}		z^{-n-\half}.
	\end{align*}
	
We observe that as $Y(z)$ acts on $v_0$, only the components $Y(n)$ with $n<0$ have non-zero action. Hence $e_0 = Y(1)$ and $f_1 = Y(0)$ annihilate $v_0$.
	
	Similarly,
	\begin{align*}
	X(z)\cdot v_0	&=	E^+_+(z)\cdot1\otimes\sum\limits_{n\in\Z+\half}A(n)\cdot
					\lrpar{\ub{-1/2}\wedge\ub{3/2}\wedge\ub{5/2}\wedge\dots}
					z^{-n-\half}\otimes e^\al z^{\frac{-\al}2}\cdot1\\
				&=	E^+_+(z)\cdot1\otimes\sum\limits_{n\le\nxhalf3}\lrpar{n-\half}\ub{n}
					\wedge\lrpar{\ub{-1/2}\wedge\ub{3/2}\wedge\ub{5/2}\wedge\ldots}
					\otimes e^\al	z^{-n-\half}.
	\end{align*}
	
	Thus, $e_1\cdot v_0 = X(0)\cdot v_0 = 0$. Furthermore,
	$$f_0\cdot v_0 = -2\lrpar{1\otimes\ub{-3/2}\wedge\ub{-1/2}\wedge\ub{3/2}\wedge\ub{5/2}\wedge\ldots\otimes e^\al}.$$
	Letting $Y(z)$ act on $f_0\cdot v_0$, we easily see that $e_0\cdot\lrpar{f_0\cdot v_0} = -2v_0$.
	Therefore, we have
	$h_0\cdot v_0	=	\lrpar{e_0f_0 - f_0e_0}\cdot v_0
				=	e_0\cdot\lrpar{f_0\cdot v_0}	= -2v_0$, and
	$h_1\cdot v_0	=	\lrpar{e_1f_1 - f_1e_1} \cdot v_0 = 0$.
Using \eqref{cliff alg deg} and \eqref{gp alg deg}, we observe that $d\cdot v_0 = 0$. Hence, the weight of $v_0$ is $\lambda_0 = -2\Lambda_0$.
	
Also,
	$Y(z)\cdot v_1			=	E^-_+(z)\otimes\sum\limits_{n\le\nxhalf3}\lrpar{n-\half}
					\partial_{\ub {-n}}\lrpar{\ub{-1/2}\wedge\ub{3/2}\wedge\ub{5/2}
					\wedge\dots}\otimes e^{-2\al}		z^{-n-\xhalf3}$, and
	$X(z)\cdot v_1	
				=	E^+_+(z)\cdot1\otimes\sum\limits_{n\le\nxhalf3}\lrpar{n-\half}\ub{n}
					\wedge\lrpar{\ub{-1/2}\wedge\ub{3/2}\wedge\ub{5/2}\wedge\dots}
					\otimes 1\	z^{-n+\half}$.

By an argument similar to the $v_0$ case, we see that $e_0$ and $e_1$ annihilate $v_1$. We also observe that $f_0$ annihilates $v_1$.  However,
	$f_1\cdot v_1 = -2\lrpar{1\otimes\ub{-1/2}\wedge\ub{5/2}\wedge\ub{7/2}\dots\otimes e^{-2\al}}.$
Letting $X(z)$ act on $f_1\cdot v_1$, we observe that $e_1\cdot\lrpar{f_1\cdot v_1} = -2v_1$.
Thus, $h_0\cdot v_1	=	\lrpar{e_0f_0 - f_0e_0}\cdot v_1 = 0$, and
	$h_1\cdot v_1	=	\lrpar{e_1f_1 - f_1e_1}\cdot v_1
				= e_1\cdot\lrpar{f_1\cdot v_1} = -2v_1.$
Using \eqref{cliff alg deg} and \eqref{gp alg deg}, we see that $d\cdot v_1 = \nhalf v_1$. Hence, the weight of $v_1$ is $\lambda_1 = -2\Lambda_1 - \half\delta$.
\end{proof}

\begin{rem}
Let $L(-2\Lambda_0)$ and $L(-2\Lambda_1-\half\delta)$ be the irreducible $\slhat$-submodules generated by the highest weight vectors $v_0$ and $v_1$, respectively. Unlike in the level 2 case (see \cite{Jing96}), $L(-2\Lambda_0)\oplus L(-2\Lambda_1-\half\delta)\subsetneq V$.
\end{rem}

\section{Representation of the \bZ-algebra}\label{Zalgebra}

Recall that the $\bZ$-algebra $\Zalg$ associated with $\slhat$ is generated by the operators\\
$\lrbrace{\bZ^\pm(m) \vert m\in\Z}$ which satisfy the relations:
\begin{align}
	\gencom{\bZ^\pm(z)}{\bZ^\pm(w)}	&= 0,		\label{plusplus}	\\
	\gencom{\bZ^+(z)}{\bZ^-(w)}		&= H(0)\dirac - 2w\partial_w\dirac,	\label{plusminus}
\end{align}
where $\bZ^\pm(z) = \sum\limits_{m\in\Z}\bZ^\pm(m) z^{-m}$ and $\gencom{\ }{\ }$ is the generalized commutator defined in \eqref{GenComDef}. We define a representation of the $\bZ$-algebra $\Zalg$ on the vacuum space $\Omega(V) = \Lambda_{\half,\widehat\nhalf}\otimes\C\lrbrack{\Z\al}$.

\begin{thm}\label{Zalg Representation}
Define the map $\pi_\Omega:\Zalg\longrightarrow End(\Omega(V))$ by
	\begin{align*}
		\bZ^+(z)	&\mapsto	A(z)e^\al z^{\frac{-\al}{2}}	\\
		\bZ^-(z)	&\mapsto	A^*(z)e^{-\al} z^{\frac{\al(0)}{2}}.
	\end{align*}
The map $\pi_\Omega$ is a representation of $\Zalg$ on $\Omega(V)$.
\end{thm}

\begin{proof}
	It suffices to show that the map $\pi_\Omega$ preserves the relations \eqref{plusplus} and \eqref{plusminus}.
	
	First, by  \eqref{aastarcontract} we have  $\pi_\Omega(\bZ^+(z))\pi_\Omega(\bZ^+(w))
		=\	:\!A(z)A(w)\!: e^{2\al} z^{-1}(zw)^{-\frac{\al(0)}{2}}$. Hence
\begin{align*}\gencom{\pi_\Omega(\bZ^+(z))}{\pi_\Omega(\bZ^+(w))}
				&=\	:\!A(z)A(w)\!:e^{2\al}(zw)^{-\frac{\al(0)}{2}}w^{-1}\dirac	\\
				&=	0 = \pi_\Omega\lrpar{\gencom{\bZ^+(z)}{\bZ^+(w)}}.
\end{align*}

Similarly, $\gencom{\pi_\Omega(\bZ^-(z))}{\pi_\Omega(\bZ^-(w))}
	=	0	= \pi_\Omega\lrpar{\gencom{\bZ^-(z)}{\bZ^-(w)}}$.

Now using \eqref{aastarcontract}  we have,

\noindent$\ds\pi_\Omega(\bZ^+(z))\pi_\Omega(\bZ^-(w))	
				=	A(z)e^\al z^{\frac{-\al}{2}}A^*(w)e^{-\al} w^{\frac{\al(0)}{2}}
				=	\lrpar{:\!A(z)A^*(w)\!: -	\frac{2zw}{(z-w)^3}}z
					\lrpar{\frac{w}{z}}^\frac{\al}{2}.$
		
Hence,
\begin{align*}
\gencom{\pi_\Omega(\bZ^+(z))}{\pi_\Omega(\bZ^-(w))}
				&=\ :\!A(z)A^*(w)\!:\lrpar{(z-w)+(w-z)} \lrpar{\frac{w}{z}}^\frac{\al}{2}	\\
				&= -2\lrpar{\frac{zw}{(z-w)^2} - \frac{zw}{(w-z)^2}}
					\lrpar{\frac{w}{z}}^\frac{\al}{2}	\\	
				&=	-2w\partial_w\lrbrack{\dirac}\lrpar{\frac{w}{z}}^\frac{\al}{2}	\\	
				&=	2w\partial_w\lrbrack{\lrpar{\frac{w}{z}}^\frac{\al}{2}}\dirac	
				-2w\partial_w\lrbrack{\lrpar{\frac{w}{z}}^\frac{\al}{2}\dirac}	\\
				&=	H(0)\dirac	-	2w\partial_w\dirac	
					= \pi_\Omega\lrpar{\gencom{\bZ^+(z)}{\bZ^-(w)}}.
\end{align*}

Thus, $\pi_\Omega$ is a representation of the algebra $\Zalg$ on $\Omega(V)$.
\end{proof}

\begin{rem}
	We note that \thmref{Zalg Representation} gives another proof that $\pi$ is a representation
of $\slhat$ on $V$ by the general theory of $\bZ$-algebras. 
\end{rem}

\begin{rem}\label{character of Omega V}
	It follows from Corollary \ref{character of V} and \thmref{Zalg Representation} that
	$$\dim_q(\Omega(V)) =
			\prod\limits_{m\ge0}\lrpar{1+q^m}^2\sum\limits_{p\in\Z}q^{\frac{p^2}{2}}.$$
\end{rem}
\vspace{5pt}
\begin{center}
Acknowledgements
\end{center}

We thank James Lepowsky for valuable comments on an earlier version
of this paper. NJ acknowledges support from
Max-Planck Institut f\"ur Mathematik, Simons Foundation grant 198129, and NSFC grants 10728102 and 11271138
during this work.
KCM acknowledges the partial support through Simons Foundation
grant 208092 and NSA grant H98230-12-1-0248.

\bibliographystyle{elsarticle-num}

\end{document}